\documentclass[letterpaper, 11 pt, conference]{IEEEtran} 
    \usepackage[left=19.1mm, right=19.mm, top=25.4mm, bottom=19.1mm]{geometry}                                                     

\usepackage{amsmath,amsfonts,amssymb,amsbsy, amsthm,ae,aecompl}
\usepackage{algorithm,algpseudocode}
\usepackage[utf8]{inputenc}
\usepackage[english]{babel}
\usepackage{bm}
\usepackage{subcaption}
\usepackage{color}
\usepackage[noadjust]{cite}
\usepackage{epsfig}
\usepackage{enumerate}
\usepackage{float}
\usepackage{fancyhdr}
\usepackage[T1]{fontenc}
\usepackage{soul}
\usepackage[acronym,toc,shortcuts]{glossaries}
\usepackage{lipsum}
\usepackage{dsfont}
\usepackage{transparent}
\usepackage{tikz,pgfplots}
\usepackage{times}
\usepackage{verbatim}
\usepackage[all]{xy}

\usepackage{etoolbox}
\AtBeginEnvironment{align}{\setcounter{subeqn}{0}}
\newcounter{subeqn} %

\usetikzlibrary{calc}
\usetikzlibrary{arrows,decorations.markings}
\pgfplotsset{
  grid style = {
    dash pattern = on 0.025mm off 0.95mm on 0.025mm off 0mm, 
    line cap = round,
    black,
    line width = 0.5pt
  },
  tick label style={font=\small},
  label style={font=\small},
  legend style={font=\footnotesize},
}

\pgfplotscreateplotcyclelist{laneas-energysto}{
cyan!60!black,		densely dotted, 	every mark/.append style={fill=cyan!80!black},mark=otimes*\\%
lime!80!black,		densely dashed,	every mark/.append style={fill=lime},mark=otimes*\\%
red,                solid, every mark/.append style={fill=red!80!black},mark=otimes*\\%
cyan!60!black, 		densely dotted,	every mark/.append style={fill=cyan!80!black},mark=diamond*\\%
lime!80!black, 		densely dashed,	every mark/.append style={fill=lime},mark=diamond*\\%
red,           		solid, every mark/.append style={solid,fill=red!80!black},mark=diamond*\\%
teal,				every mark/.append style={fill=teal!80!black},mark=*\\%
orange,				every mark/.append style={fill=orange!80!black},mark=square*\\%
red!70!white,		mark=star\\%
yellow!60!black,		densely dashed, every mark/.append style={solid,fill=yellow!80!black},mark=square*\\%
black,				every mark/.append style={solid,fill=gray},mark=otimes*\\%
blue,				densely dashed,mark=star,every mark/.append style=solid\\%
red,					densely dashed,every mark/.append style={solid,fill=red!80!black},mark=diamond*\\%
}

\graphicspath{{figures/}}

\bgroup
\makeglossaries

\newacronym{DGG}{DGG}{deterministic geometric graph}
\newacronym[sort=ell]{test}{{LED}}{limiting eigenvalue distribution}
\newacronym{LSD}{LSD}{limiting spectral distribution}
\newacronym{ESDF}{ESDF}{ empirical spectral distribution function}
\newacronym{RGG}{RGG}{random geometric graph}
\newacronym{RGGs}{RGGs}{random geometric graphs}
\newacronym{BS}{BS}{base station}
\newacronym{BSs}{BSs}{base stations}
\newacronym{PPP}{PPP}{Poisson point process}
\newacronym{PPPs}{PPPs}{Poisson point processes}
\newacronym{SINR}{SINR}{signal-to-interference-plus-noise ratio }
\newacronym{SE}{SE}{spectral efficiency}
\newacronym{DFT}{DFT}{discrete Fourier transform}
\newacronym{SDMA}{SDMA}{space-division multiple Access }
\newacronym{PDF}{PDF}{probability density function}
\newacronym{ZF}{ZF}{zero-forcing}

\newacronym{UE}{UE}{user}
\newacronym{UEs}{UEs}{users}

\newacronym{MRC}{MRC}{maximum ratio combining}
\newacronym{CSI}{CSI}{channel state information}
\newacronym{RF}{RF}{radio frequency}

\newacronym{ADMM}{ADMM}{Alternating Direction Method of Multipliers}
\newacronym{APC}{APC}{area power consumption}
\newacronym{ASE}{ASE}{area spectral efficiency}
\newacronym{ER}{ER}{Erd\"{o}s-Rényi}
\newacronym{CDN}{CDN}{content delivery network}
\newacronym{CN}{CN}{core network}
\newacronym{ICN}{ICN}{information-centric network}
\newacronym{CF}{CF}{collaborative filtering}
\newacronym{CRP}{CRP}{{C}hinese restaurant process}
\newacronym{CS}{CS}{central scheduler}
\newacronym{D2D}{D2D}{device-to-device}
\newacronym{EE}{EE}{energy efficiency}
\newacronym{ICIC}{ICIC}{inter-cell interference coordination}
\newacronym{LTE}{LTE}{long term evolution}
\newacronym{MIMO}{MIMO}{multiple-input multiple-output}

\newacronym{SBS}{SBS}{small base station}

\newacronym{SCN}{SCN}{small cell network}
\newacronym{SVD}{SVD}{singular value decomposition}
\newacronym{UT}{UT}{user terminal}
\newacronym{QoS}{QoS}{quality-of-service}
\newacronym{QoE}{QoE}{quality-of-experience}
\newacronym{RAN}{RAN}{radio access network}

\newacronym{PGFL}{PGFL}{probability generating functional}
\newacronym{HetNet}{HetNet}{heterogeneous network}
\newtheorem{definition}{Definition}
\newtheorem{theorem}{Theorem}
\newtheorem{lemma}{Lemma}

\begin{document}
\small
\title{Spectral Analysis of the Adjacency Matrix of Random Geometric Graphs}
\author{
		\IEEEauthorblockN{Mounia Hamidouche$^{\star}$, Laura Cottatellucci$^{\dagger}$, Konstantin Avrachenkov$^{ \diamond}$}
		\IEEEauthorblockA{
				\vspace{-0.0cm}
				\\
				\small  $^{\star}$ Departement of Communication Systems, EURECOM, Campus SophiaTech, 06410 Biot, France \\	
				$^{\dagger}$ Department of Electrical, Electronics, and Communication Engineering, FAU, 51098 Erlangen, Germany \\
				$^{\diamond}$ Inria, 2004 Route des Lucioles, 06902 Valbonne, France \\
				$^{}$ \\
	mounia.hamidouche@eurecom.fr, laura.cottatellucci@fau.de, k.avratchenkov@inria.fr. \\
				\vspace{-0.85cm}
		}
		\thanks{This research was funded by the French Government through the Investments for
the Future Program with Reference: Labex UCN@Sophia-UDCBWN.
}
}
\IEEEoverridecommandlockouts
\maketitle
\begin{abstract}

In this article, we analyze the limiting eigenvalue distribution (LED) of random geometric graphs (RGGs). The RGG is constructed by uniformly distributing $n$ nodes on the $d$-dimensional torus $\mathbb{T}^d \equiv [0, 1]^d$ and connecting two nodes if their $\ell_{p}$-distance, $p \in [1, \infty]$ is at most $r_{n}$.   In particular, we study the LED of the adjacency matrix of  RGGs in the connectivity regime, in which the average vertex degree scales as $\log\left( n\right)$ or faster, i.e., $\Omega \left(\log(n) \right)$.  In the connectivity regime and under some conditions on the radius $r_{n}$, we show that the LED of the adjacency matrix of RGGs converges to the LED of the adjacency matrix of a deterministic geometric graph (DGG) with nodes in a grid as $n$ goes to infinity. Then, for $n$ finite, we use the structure of the DGG to approximate the eigenvalues of the adjacency matrix of the RGG and provide an upper bound for the approximation error. 
\end{abstract}

\section{\bf Introduction }
\label{sec:introduction}
In recent years, random graph theory has been applied to model many complex real-world phenomena. A basic random graph used to model complex networks is the \gls{ER} graph \cite{erdos1959random}, where edges between the nodes appear with equal probabilities.  In \cite{gilbert1961random}, the author introduces another random graph called \gls{RGG} where nodes have some random position in a metric space and the edges are determined by the position of these nodes.  Since then, \gls{RGG}  properties have been widely studied \cite{penrose2003random}.

 \glspl{RGG} are very useful to model problems in which the geographical distance is a critical factor.  For example, {\glspl{RGG}} have been applied to wireless communication network \cite{bettstetter2002minimum}, sensor network \cite{yick2008wireless} and to study the dynamics of a viral spreading in a specific network of interactions  \cite{preciado2009spectral}, \cite{ganesh2005effect}.  Another motivation for \glspl{RGG} in arbitrary dimensions is multivariate statistics of high-dimensional data. In this case, the coordinates of the nodes can represent the attributes of the data. Then, the metric imposed by the \gls{RGG} depicts the similarity between the data.

In this work,  the \gls{RGG} is constructed by considering a finite set $\mathcal{X}_{n}$ of $n$ nodes, $x_{1},...,x_{n},$ distributed uniformly and independently on the $d$-dimensional torus $\mathbb{T}^d \equiv [0, 1]^d$. We choose a torus instead of a cube in order to avoid boundary effects. Given a geographical distance, $r_{n} >0 $, we form a graph by connecting two nodes $x_{i}, x_{j} \in \mathcal{X}_{n}$  if their $\ell_{p}$-distance, $p \in [1, \infty]$ is at most $r_{n}$, i.e., $\|x_{i}-x_{j} \|_{p} \leq r_{n}$, where $\|.\|_{p}$ is the $\ell_{p}$-metric defined as
\begin{equation*}
\small
\label{eqq1}
\| x_{i}- x_{j} \|_{p} =\left\{
\begin{array}{ll}
 \left(  \sum_{k=1}^d \vert x_{i}^{(k)}-x_{j}^{(k)}\vert^p \right)^{1/p} &   p \in [1, \infty),\\ 
 \max\lbrace \vert x_{i}^{(k)}-x_{j}^{(k)}\vert, \ k \in [1, d]  \rbrace & p=\infty.
    \end{array}
\right. 
\end{equation*}

The \gls{RGG} is denoted by $G(\mathcal{X}_{n}, r_{n})$.  Note that for the case $p=2$ we obtain the Euclidean metric on $\mathbb{R}^d$.   Typically, the function $r_{n}$ is chosen such that $r_{n}\rightarrow 0$ when $n \rightarrow \infty$.

The degree of a vertex in $G(\mathcal{X}_{n}, r_{n})$ is the number of edges connected to it.  The average vertex degree in $G(\mathcal{X}_{n}, r_{n})$ is given by \cite{penrose2003random}
\begin{equation*}
a_{n} = \theta^{(d)} nr_{n}^d,
\end{equation*} 
where $\theta^{(d)}=\pi^{d/2}/\Gamma(d/2+1)$ denotes the volume of the $d$-dimensional unit hypersphere in $\mathbb{T}^d$ and $\Gamma(.)$ is the Gamma function. 

Different values of $r_{n}$, or equivalently $a_{n}$, lead to different geometric structures in \glspl{RGG}. In \cite{penrose2003random}, different interesting regimes are introduced: the \textit{connectivity regime} in which $a_{n}$ scales as $\log(n)$ or faster, i.e., $\Omega(\log(n))$\footnote{The notation $f(n) =\Omega(g(n))$ indicates that $f(n)$ is bounded below by $g(n)$ asymptotically, i.e., $\exists K>0$ and $ n_{o} \in \mathbb{N}$ such that $\forall n > n_{0}$ $f(n) \geq K g(n)$.},  the  \textit{thermodynamic regime} in which $a_{n}\equiv \gamma$, for $\gamma >0$ and the \textit{dense regime}, i.e., $a_{n}\equiv \Theta(n).$

\glspl{RGG} can be described by a variety of random matrices such as adjacency matrices, transition probability matrices and normalized Laplacian. The spectral properties of those random matrices are powerful tools to predict and analyze complex networks behavior.  In this work, we  give a special attention to the \gls{test} of the adjacency matrix of \glspl{RGG}  in the connectivity regime.


Some works analyzed the spectral properties of \glspl{RGG} in different regimes. In particular, in the thermodynamic regime, the authors in \cite{bordenave2008eigenvalues}, \cite{blackwell2007spectra} show that the spectral measure of the adjacency matrix of \glspl{RGG} has a limit as $n \to \infty$. However, due to the difficulty to compute exactly this spectral measure,  Bordenave in \cite{bordenave2008eigenvalues} proposes an approximation for it as $\gamma \rightarrow \infty$.

 In the connectivity regime, the work in \cite{preciado2009spectral} provides a closed form expression for the asymptotic spectral moments of the adjacency matrix of $G(\mathcal{X}_{n}, r_{n})$. Additionnaly,  Bordenave in \cite{bordenave2008eigenvalues} characterizes the spectral measure of the adjacency matrix normalized by $n$ in the dense regime. However, in the connectivity regime and as $n\rightarrow \infty$,  the normalization factor $n$ puts to zero all the eigenvalues of the adjacency matrix that are finite and only the infinite eigenvalues in the adjacency matrix are nonzero in the normalized adjacency matrix. Motivated by this results, in this work we analyze the behavior of the eigenvalues of the adjacency matrix  without normalization in the connectivity regime and in a wider range of the connectivity regime.

First, we propose an approximation for the actual \gls{test} of the \gls{RGG}. Then, we provide a bound on the Levy distance between this approximation and the actual distribution. More precisely, for $\epsilon>0$ we show that the \glspl{test} of the adjacency matrices of the \gls{RGG} and the \gls{DGG} with nodes in a grid converge to the same limit when $a_{n}$ scales as $\Omega (\log^{\epsilon}(n)\sqrt{n})$ for $d=1$ and as $\Omega (\log^{2}(n))$ for $d\geq 2$. 
Then, under the $\ell_{\infty}$-metric we provide an analytical approximation for the eigenvalues of the adjacency matrix of \glspl{RGG} by taking the $d$-dimensional discrete Fourier transform (DFT) of an $n=\mathrm{N}^{d}$ tensor of rank $d$ obtained from the first block row of the adjacency matrix of the \gls{DGG}.

The rest of this paper is organized as follows.  In Section \ref{sec:systemmodel} we describe the model, then we present our main results on the concentration of the \gls{test} of large \glspl{RGG} in  the connectivity  regime.  Numerical results are given in Section \ref{sec:results} to validate the theoretical results. Finally, conclusions are given in Section \ref{sec:conclusion}.
\section{\bf\bf Spectral Analysis of \glspl{RGG}}
\label{sec:systemmodel} 
To study the spectrum of $G(\mathcal{X}_{n}, r_{n})$ we introduce an auxiliary graph called the \gls{DGG}. The \gls{DGG} denoted by $G(\mathcal{D}_{n}, r_{n})$ is formed by letting $\mathcal{D}_{n}$ be the set of $n$ grid points that are at the intersections of axes parallel hyperplanes with separation $n^{-1/d}$, and connecting two points $x'_{i}$, $x'_{j}$ $\in \mathcal{D}_{n}$ if $\|x'_{i}-x'_{j} \|_{p} \leq r_{n}$ with $p\in [1, \infty]$. Given two nodes, we assume that there is always at most one edge between them. There is no edge from a vertex to itself. Moreover,  we assume that the edges are not directed. 
 
  Let $\mathbf{A}(\mathcal{X}_{n})$  be the  adjacency matrix of $G(\mathcal{X}_{n}, r_{n})$, with entries 
\begin{equation*}
\mathbf{A}(\mathcal{X}_{n})_{i j}=\chi [x_{i} \sim x_{j}],
\label{RW}
\end{equation*}
where the term $\chi[x_{i}\thicksim x_{j}] $ takes the value 1 when there is a connection between nodes $x_{i}$ and $x_{j}$ in $G(\mathcal{X}_{n}, r_{n})$ and zero otherwise, represented as
\begin{equation*}
\label{eqq1}
\chi[x_{i}\thicksim x_{j}] =\left\{
    \begin{array}{ll}
        1, &  \| x_{i} - x_{j}\|_{p} \leq r_{n}, \ \ i \neq j,  \ \ p \in [1, \infty]\\
       0, & \mathrm{otherwise}.
    \end{array}
\right.
\end{equation*}

A similar definition holds for $\mathbf{A}(\mathcal{D}_{n})$ defined over $G(\mathcal{D}_{n}, r_{n})$. The matrices $\mathbf{A}({\mathcal{X}_{n}})$  and $\mathbf{A}({\mathcal{D}_{n}})$ are symmetric and their spectrum consists of real eigenvalues. We denote by $\lbrace \lambda_{i}, i=1,..,n \rbrace$ and $\lbrace \mu_{i}, i=1,..,n \rbrace$ the sets of all real eigenvalues of the real symmetric square matrices $\mathbf{A}({\mathcal{D}_{n}})$  and $\mathbf{A}({\mathcal{X}_{n}})$ of order $n$, respectively.
The empirical spectral distribution functions $v_{n}(x)$ and $v'_{n}(x)$ of the adjacency matrices of an \gls{RGG} and a \gls{DGG}, respectively  are defined as
\begin{equation*}
v_{n}(x)=\dfrac{1}{n} \sum\limits_{i=1}^n \mathrm{I}\lbrace \mu_{i} \leq x\rbrace \ \ \mathrm{and }\  \ \ v'_{n}(x)=\dfrac{1}{n} \sum\limits_{i=1}^n \mathrm{I}\lbrace \lambda_{i}\leq  x\rbrace,
\end{equation*}
where $\mathrm{I}\lbrace \mathrm{B} \rbrace$ denotes the indicator of an event $\mathrm{B}$.

Let $a'_{n}$ be the degree of the nodes in $G(\mathcal{D}_{n}, r_{n})$. In the following Lemma \ref{bound}  we provide an upper bound for $a'_{n}$ under any $\ell_{p}$-metric.

\begin{lemma}
\label{bound}
 For any chosen $\ell_{p}$-metric with $p \in [1,\infty]$ and $d\geq1$, we have
\begin{equation*}
a'_{n} \leq d^{\frac{1}{p}}2^{d}a_{n}  \left( 1+\frac{1}{2a_{n}^{1/d}}\right)^d.
\end{equation*}
\end{lemma}
\begin{proof}
See Appendix \ref{app:bound}
\end{proof}

To prove our result on the concentration of the \gls{test} of \glspl{RGG} and investigate its relationship with the \gls{test} of \glspl{DGG} under any $\ell_{p}$-metric, we use the Levy distance between two distribution functions defined as follows.
 
\begin{definition}(\cite{taylor2012introduction}, page 257)
Let $v_{n}^A$ and $v_{n}^B$ be two distribution functions on $\mathbb{R}$. The Levy distance $L(v_{n}^A, v_{n}^B)$ between them is the infimum of all positive $\epsilon$ such that, for all $x \in \mathbb{R}$ 
\begin{equation*}
v_{n}^A(x-\epsilon) -\epsilon \leq v_{n}^B(x)\leq v_{n}^A(x+\epsilon)+\epsilon.
\end{equation*}
\end{definition}

\begin{lemma}(\cite{bai2008methodologies}, page 614)
\label{Difference Inequality}
Let A and B be two $n$ $\times$ $n$ Hermitian matrices with eigenvalues $\lambda_{1},...,\lambda_{n}$ and $\mu_{1},...,\mu_{n}$, respectively. Then
\begin{equation*}
L^{3}(v_{n}^A, v_{n}^B) \leqslant \dfrac{1}{n}tr(A-B)^2,
\end{equation*}
where $L(v_{n}^{A},v_{n}^{B})$ denotes the Levy distance between the empirical distribution functions $v_{n}^{A}$ and $v_{n}^{B}$ of the eigenvalues of $A$ and $B$, respectively.
\end{lemma}


Let $\mathrm{M}_{n}$ be the minimum bottleneck matching distance corresponding to the minimum length such that there exists a perfect matching of the random nodes to the grid points for which the distance between every pair of matched points is at most $\mathrm{M}_{n}$.

Sharp bounds for $\mathrm{M}_{n}$ are given in \cite{shor1991minimax}\cite{leighton1986tight}\cite{goel2004sharp}. We repeat them in the following lemma for convenience.
\begin{lemma}
Under any $\ell_{p}$-norm, the bottleneck matching is 
 \begin{itemize}
 \item $\mathrm{M}_{n} = O \left( \left( \dfrac{\log n}{n}\right)^{1/d}\right),$ \ \ when $d \geq 3$ \cite{shor1991minimax}.
 \vspace{0.1cm}
 
  \item $\mathrm{M}_{n} = O \left(\left(\dfrac{\log^{3/2} n}{n}\right)^{1/2} \right),$ \ \ when $d= 2$ \cite{leighton1986tight}.
   \vspace{0.1cm}
  
    \item $\mathrm{M}_{n} = O \left(\sqrt{\dfrac{\log \epsilon^{-1}}{n}}\right),$ with prob. $\geq 1-\epsilon$,\ $d= 1$  \cite{goel2004sharp}.
 \end{itemize}
\end{lemma}

Under the condition $\mathrm{M}_{n}= o(r_{n})$, we provide an upper bound for the Levy distance between $v_{n}$ and $v'_{n}$ in the following lemma.

\begin{lemma}
\label{lemafirst}
For $d \geq 1$, $p \in [1, \infty]$ and $\mathrm{M}_{n}= o(r_{n})$, the Levy distance between $v_{n}$ and $v'_{n}$ is upper bounded as
\begin{equation}
\begin{aligned}
\label{equation0}
& L^{3} \left( v_{n}, v'_{n} \right)  \leq d^{\frac{1}{p}}  2^{d+1} \left|  \frac{1}{n}\sum\limits_{\substack{ i}}^{} \mathbf{N}(x_{i}) - a_{n} \right|  \\
 &+d^{\frac{1}{p}} 2^{d+1}  \left| a_{n} -\frac{2}{n}  \sum\limits_{\substack{ i}}^{} \mathrm{L}_{i} \right| +a'_{n},
\end{aligned}
\end{equation}
where,  $\mathbf{N}(x_{i})$ denotes the degree of $x_{i}$ in $G(\mathcal{X}_{n}, r_{n})$ and $\mathrm{L}_{i} \sim \mathrm{Bin}\left(n, \theta^{(d)}\left( r_{n}-2\mathrm{M}_{n}\right) \right)$. \comment{Due to space constraints we omit the details of the proof}. 
\end{lemma}
\begin{proof}
See Appendix  \ref{app:lemmafirst}.
\end{proof}

The condition enforced on $r_{n}$, i.e., $\mathrm{M}_{n}= o(r_{n})$ implies that for $\epsilon >0$, (\ref{equation0}) holds when $a_{n}$ scales as $\Omega (\log^{\epsilon}(n)\sqrt{n})$ for $d=1$, as $\Omega (\log^{\frac{3}{2}+\epsilon}(n))$ for $d= 2$ and as $\Omega (\log^{1+\epsilon}(n))$ for $d\geq 3$.

In what follows, we show that the \gls{test} of the adjacency matrix of $G(\mathcal{X}_{n}, r_{n})$ concentrate around the \gls{test} of the adjacency matrix of $G(\mathcal{D}_{n}, r_{n})$ in the connectivity regime in the sense of convergence in probability.

Notice that the term $ \sum\limits_{\substack{ i}}^{} \mathbf{N}(x_{i}) / 2$  in Lemma \ref{lemafirst} counts the number of edges in $G(\mathcal{X}_{n}, r_{n})$. For convenience, we denote $ \sum\limits_{\substack{ i}}^{} \mathbf{N}(x_{i}) / 2$ as $\xi_{n}$.  To show our main result we apply the Chebyshev inequality given in Lemma \ref{Chebyschev-inequality} on the random variable $\xi_{n}$. For that, we need to determine $\mathrm{Var}(\xi_{n})$

\begin{lemma}(Chebyshev Inequality)
\label{Chebyschev-inequality}
Let $\mathrm{X}$ be a random variable with an expected value $\mathbb{E}\mathrm{X}$ and a variance $\mathrm{Var}\left(\mathrm{X}\right)$. Then, for any $t>0$
\begin{equation*}
\mathbb{P}  \lbrace \vert \mathrm{X}-\mathbb{E}\mathrm{\mathrm{X}} \vert  \geq t \rbrace \leq \dfrac{\mathrm{Var}(\mathrm{X})}{t^2}.
\end{equation*}
\end{lemma}

\begin{lemma}
\label{variance-edges}
When $x_{1},...,x_{n}$ are i.i.d. uniformly distributed in the $d$-dimensional unit torus $\mathbb{T}^{d}=[0, 1]$
\begin{equation*}
\mathrm{Var} \left(\xi_{n}\right) \leq [\theta^{(d)}+2\theta^{(d)} a_{n} ].
\end{equation*}
\end{lemma}
\begin{proof}
The proof follows along the same lines of Proposition A.1 in \cite{muller2008two} when extended to a unit torus and applied to i.i.d. and uniformly distributed nodes.
\end{proof}

We can now state the main theorem on the concentration of the adjacency matrix of $G(\mathcal{X}_{n}, r_{n})$.
 
\begin{theorem}
\label{theorem-connectivity} 
For  $d \geq 1$, $p\in[1, \infty]$, $a \geq 1$, $\mathrm{M}_{n}= o(r_{n})$ and  $t >0$,  we have
\begin{equation*}
\mathbb{P} \lbrace L^{3} \left( v_{n}, v'_{n} \right) >t\rbrace  \leq   2n \exp\left( \dfrac{-a_{n}\varepsilon^2}{3} \left(1-\dfrac{2\mathrm{M}_{n}}{r_{n}}\right)   \right)
\end{equation*}
\begin{equation*}
 + \frac{n \left[\theta^{(d)}(r_{n}-2\mathrm{M}_{n})(a-1)+1\right]^n}{a^{\left(\frac{t}{d^{\frac{1}{p}}2^{d+3}}+\frac{a_{n}(2-c)}{4}\right)}}
\end{equation*}
\begin{equation}
\ \ \ \ +\frac{d^{\frac{2}{p}} 2^{2d+6}\left[ \theta^{(d)}+2\theta^{(d)} a_{n} \right]}{n^2t^2},
\end{equation}
where $\varepsilon= \left( \frac{t}{d^{\frac{1}{p}}2^{d+2}a_{n}}+\dfrac{(2-c)}{4}- \frac{2\mathrm{M}_{n}}{r_{n}} \right)$ and $c=\left(1+\frac{1}{2a_{n}^{1/d}} \right)^d$.
\\

In particular,  for every $t>0$, $a \geq 2$, $\epsilon>0$ and $a_{n}$ that scales as $\Omega (\log^{\epsilon}(n)\sqrt{n})$ when $d=1$, as $\Omega (\log^{2}(n))$ when $d\geq 2$, we have

\begin{equation*}
\lim_{n \to\infty} \mathbb{P} \left\lbrace L^{3} \left( v_{n}, v'_{n} \right) >t  \right\rbrace  = 0.
\end{equation*}
\end{theorem}
\begin{proof}
  See Appendix \ref{app:theorem-connectivity}.
\end{proof}
This result is shown in the sense of convergence in probability by a straightforward application of Lemma  \ref{Chernoff} and \ref{Chernof-binomial} on the random variable $\mathrm{L}_{i}$, then by applying Lemma \ref{Chebyschev-inequality} and \ref{variance-edges} to $\xi_{n}.$

In what follows, we provide the eigenvalues of $\mathbf{A}(\mathcal{D}_{n})$ which approximates the eigenvalues of $\mathbf{A}(\mathcal{X}_{n})$ for $n$ sufficiently large.

\begin{lemma}
\label{corollary1:density}
 For $d \geq 1$ and using the $\ell_{\infty}$-metric, the eigenvalues of $\mathbf{A}(\mathcal{D}_{n})$ are given by
 \begin{equation}
 \label{equation5}
\lambda_{m_{1},...,m_{d}}= \prod_{s=1}^d  \dfrac{\sin(\frac{m_{s} \pi}{\mathrm{N}}(a'_{n}+1)^{1/d})}{\sin(\frac{m_{s} \pi}{\mathrm{N}})}  -1,
 \end{equation}
where,  $m_{1},...,m_{d}$ $\in$ $\lbrace 0,...\mathrm{N}-1 \rbrace$, $a'_{n}=(2k_{n}+1)^d-1, k_{n}= \left\lfloor \mathrm{N}r_{n} \right\rfloor \ \ \mathrm{and} \ \ n=\mathrm{N}^{d}.$ The term $ \left\lfloor x \right\rfloor$ is the integer part, i.e.,  the greatest integer less than or equal to $x$.
\end{lemma}
\begin{proof}
 See Appendix  \ref{app:lemma3}. 
\end{proof}
The proof utilizes the result in \cite{nyberg2014laplacian} which shows that the eigenvalues of the adjacency matrix of a \gls{DGG} in $\mathbb{T}^d$ are found by taking the $d$-dimensional DFT of an $\mathrm{N}^{d}$ tensor of rank $d$ obtained from the first block row of $\mathbf{A}(\mathcal{D}_{n})$.

For $\epsilon >0$, Theorem \ref{theorem-connectivity} shows that when $a_{n}$ scales as $\Omega (\log^{\epsilon}(n)\sqrt{n})$ for $d=1$ and  as $\Omega (\log^{2}(n))$ when $d\geq 2$, the \gls{test} of the adjacency matrix of an \gls{RGG} concentrate around the \gls{test} of the adjacency matrix of a \gls{DGG} as $n \rightarrow \infty$.   Therefore, for $n$ sufficiently large, the eigenvalues of the \gls{DGG} given in (\ref{equation5}) approximate very well the eigenvalues of the \gls{DGG}.

\section{\bf  Numerical Results}
\label{sec:results}
We present simulations to validate the results obtained in Section \ref{sec:systemmodel}. More specifically, we corroborate our results on the spectrum of the adjacency matrix of \glspl{RGG} in the connectivity regime by comparing the simulated and the analytical results.

Fig. \ref{figg}(a) shows the cumulative distribution functions of the eigenvalues of the adjacency matrix of an \gls{RGG} realization and the analytical spectral distribution in the connectivity regime. We notice that  for the chosen average vertex degree $a_{n}= \log(n)\sqrt{n}$ and $d=1$, the curves corresponding to the eigenvalues of the \gls{RGG} and the \gls{DGG} fit very well for a large value of $n$.

\captionsetup[figure]{labelfont={bf},labelformat={default},labelsep=period,name={Fig.}}
\begin{figure}
\begin{minipage}[b]{0.45\textwidth}
\begin{subfigure}[b]{\linewidth}
\includegraphics[width=\linewidth]{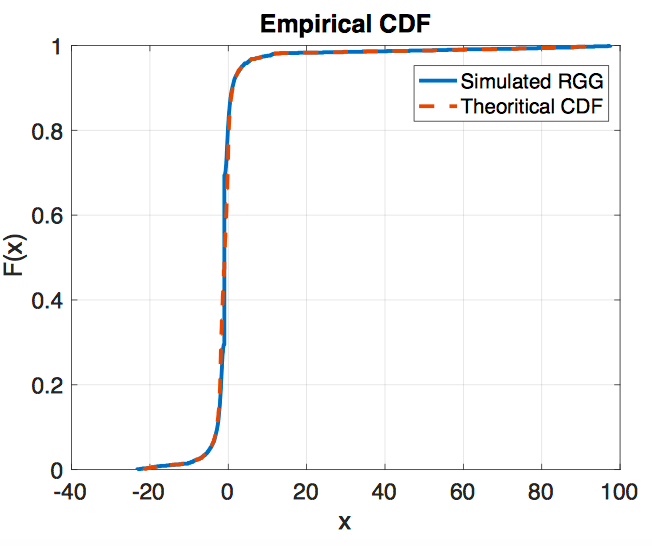}
\caption{Connectivity regime, $r_{n}=\frac{\log(n)}{\sqrt{n}}$, $n=2000$.}
\end{subfigure}

\end{minipage}
\caption{An illustration of the cumulative distribution function of the eigenvalues of an \gls{RGG}. }
\label{figg}
\end{figure}

\section{\bf Conclusion}
\label{sec:conclusion}
In this work, we study the spectrum of the adjacency matrix of \glspl{RGG} in the connectivity regime.  Under some conditions on the average vertex degree $a_{n}$, we show that the \glspl{test} of the adjacency matrices of an \gls{RGG} and a \gls{DGG} converge to the same limit as $n \rightarrow \infty$. Then,  based on the regular structure of the \gls{DGG}, we approximate the  eigenvalues of $\mathbf{A}(\mathcal{X}_{n})$ by the eigenvalues of $\mathbf{A}(\mathcal{D}_{n})$ by taking the $d$-dimensional DFT of an $\mathrm{N}^{d}$ tensor of rank $d$ obtained from the first block row of $\mathbf{A}(\mathcal{D}_{n})$.

\section{\bf Acknowledgement}
This research was funded by the French Government through the Investments for
the Future Program with Reference: Labex UCN@Sophia-UDCBWN.


\bibliographystyle{IEEEtran}
\bibliography{references}

\appendices

\section{Proof of Lemma \ref{bound} }
\label{app:bound}

In this Appendix, we upper bound the vertex degree $a'_{n}$ under any $\ell_{p}$-metric, $p\in [1,\infty]$. 

Assume that $G(\mathcal{X}_{n}, r_{n})$ and $G(\mathcal{D}_{n}, r_{n})$ are formed using the $\ell_{\infty}$-metric and let $a_{n}$ and $a'_{n}$ be their average vertex degree and vertex degree, respectively.

 In this case, for a $d$-dimensional {\gls{DGG}} with $n= \mathrm{N}^d$ nodes, the vertex degree $a'_{n}$ is given by {\cite{nyberg2014laplacian}}
\begin{equation*}
a'_{n}=(2k_{n}+1)^d-1, \ \ \mathrm{ with} \ \ k_{n}= \left\lfloor \mathrm{N}r_{n} \right\rfloor \ \ \mathrm{and} \ \ n=\mathrm{N}^{d}.
\end{equation*}

Therefore, for $\theta^{(d)} \geq 2$ and $d\geq1,$ we have
\begin{align*}
\label{equa}
a'_{n}=(2k_{n}+1)^d-1 &\leq  \frac{2^d a_{n}}{\theta^{(d)}}\left(1+\frac{1}{2n^{1/d}r_{n}} \right)^d\\
&\leq  2^d a_{n}\left(1+\frac{1}{2a_{n}^{1/d}} \right)^d.
\end{align*}

Now, let $b'_{n}$ and $b_{n}$ be the vertex degree and the average vertex  degree in $G(\mathcal{D}_{n}, r_{n})$ and $G(\mathcal{X}_{n}, r_{n})$, respectively when using any $\ell_{p}$-metric, $p \in [1, \infty]$.  Notice that for any $p \in [1, \infty],$ we have
\begin{equation*}
\| x\|_{\infty} \leq \|x \|_{p}.
\end{equation*}

Then, the number of nodes $a'_{n}$ that falls in the ball of radius $r_{n}$ is greater or equal than $b'_{n}$, i.e., $a'_{n} \geq b'_{n}$. Hence,
\begin{align*}
b'_{n} \leq a'_{n}& \leq  2^d a_{n}\left(1+\frac{1}{2n^{1/d}r_{n}} \right)^d\\
&=\dfrac{ d^{1/p} 2^d a_{n}}{d^{1/p}}\left(1+\frac{1}{2n^{1/d}r_{n}} \right)^d.
\end{align*}

It remains to show the relation between $b'_{n}$ and $b_{n}$.

 Assume that the \gls{RGG} is formed by connecting each two nodes when $d^{1/p}\Vert x_{i}-x_{j}\Vert_{\infty} \leq r_{n}.$ This simply means that the graph is obtained using the $\ell_{\infty}$-metric with a radius equal to $\frac{r_{n}}{d^{1/p}}$. Then, the average vertex degree of this graph is $\frac{a_{n}}{d^{1/p}}$. In addition, we have 
\begin{equation*}
\| x\|_{p} \leq d^{\frac{1}{p}}\|x \|_{\infty}.
\end{equation*}

Therefore, 
\begin{align*}
b'_{n} \leq d^{1/p} 2^d b_{n}\left(1+\frac{1}{2n^{1/d}r_{n}} \right)^d.
\end{align*}
 \qed

\section{Proof of Lemma \ref{lemafirst} }
\label{app:lemmafirst}
In this Appendix, we upper bound the Levy distance between the distribution functions $v_{n}$ and $v'_{n}$. 

By a straightforward application of Lemma \ref{Difference Inequality}, we have

\begin{alignat*}{4}
 L^{3} \left( v_{n}, v'_{n} \right)  &\leq \frac{1}{n} \mathrm{Trace}\left[ (\mathbf{A}(\mathcal{X}_{n})- \mathbf{A}(\mathcal{D}_{n})\right]^{2} \\
& =  \frac{1}{n } \sum\limits_{\substack{ i}}^{}{\sum\limits_{\substack{ j}}^{}{}} \left[ \chi [x_{i}\thicksim x_{j}]^{}- \chi [x'_{i}\thicksim x'_{j}]^{}\right]^2 \\
& \overset{(a)}{=}  \frac{1}{n }\sum\limits_{\substack{ i}}^{} \mathbf{N}(x_{i}) +a'_{n} -\dfrac{2}{n}  \sum\limits_{\substack{ i}}^{} \mathbf{N}(x_{i}, x'_{i})\\
& \overset{(b)}{\leq}  \frac{1}{n }\sum\limits_{\substack{ i}}^{} \mathbf{N}(x_{i}) +a'_{n} -\dfrac{2}{n}  \sum\limits_{\substack{ i}}^{} \mathrm{L}_{i}\\
& \overset{}{\leq}  \left| \frac{1}{n}\sum\limits_{\substack{ i}}^{} \mathbf{N}(x_{i})  -\dfrac{2}{n}  \sum\limits_{\substack{ i}}^{} \mathrm{L}_{i}\right|  + a'_{n}\\
 & \overset{}{\leq} d^{\frac{1}{p}} 2^{d+1} \left| \dfrac{1}{n}\sum\limits_{\substack{ i}}^{} \mathbf{N}(x_{i})  -\dfrac{2}{n}  \sum\limits_{\substack{ i}}^{} \mathrm{L}_{i}\right|  + a'_{n}\\
&  \overset{}{\leq} d^{\frac{1}{p}} 2^{d+1} \left|  \dfrac{1}{n}\sum\limits_{\substack{ i}}^{} \mathbf{N}(x_{i}) - a_{n} \right| \\
& \ \ \ \ \ \ \ \ \ \  +d^{\frac{1}{p}} 2^{d+1}  \left| a_{n} -\dfrac{2}{n}  \sum\limits_{\substack{ i}}^{} \mathrm{L}_{i} \right| +a'_{n}.
\end{alignat*}

Step $(a)$ follows from $ \mathbf{N}(x_{i})=\sum \limits_{\substack{j}}^{}{ \chi [x_{i}\sim x_{j}]}$ and $ a'_{n}= \sum\limits_{\substack{j}}^{}{\chi [x'_{i}\sim x'_{j}]} $, and by defining $\mathbf{N}(x_{i}, x'_{i})$= $\sum\limits_{\substack{ j}}^{}{\chi[x_{i}\thicksim x_{j}]\chi[x'_{i}\thicksim x'_{j}]}$. Step $(b)$ follows from noticing that when $\| x_{i}-x_{j} \|_{p} \leq r_{n}-2 \mathrm{M}_{n}$, then $\| x'_{i}-x'_{j} \|_{p} \leq r_{n}$. So, all points within a radius of $r_{n}-2 \mathrm{M}_{n}$ of $x_{i}$ map to the neighbors of $x'_{i}$ \cite{rai2007spectrum}. Thus, $ \mathbf{N}(x_{i}, x'_{i})$ is stochastically greater than  the random variable $\mathrm{L}_{i} \thicksim \mathrm{Bin}(n, \theta^{(d)}(r_{n}-2 \mathrm{M}_{n}) )$.
 \qed

\section{Proof of Theorem \ref{theorem-connectivity} }
\label{app:theorem-connectivity}

We provide an upper bound on the probability that the Levy distance between the distribution functions $v_{n}$ and  $v'_{n}$ is higher than $t>0$.  The following lemmas are useful for the following studies.

\begin{lemma}(Chernoff Bound)
\label{Chernoff}
Let $\mathrm{X}$ be a random variable. Then, for any $t>0$
\begin{equation*}
\mathbb{P}  \lbrace  \mathrm{X} \geq t \rbrace \leq \dfrac{F(a)}{a^t},
\end{equation*}
where $F(a)$ is the probability generating function and $a \geq 1.$
\end{lemma}

\begin{lemma}(\cite{janson2011random},
Corollary 2.3, page 27)
\label{Chernof-binomial}
If $\mathrm{X} \in \mathrm{Bin}(n, p)$,  $\mathbb{E}\mathrm{X}=np$ and $0 <\varepsilon \leq \frac{3}{2}$, we have
\begin{equation*}
\mathbb{P} \lbrace \vert \mathrm{X}-\mathbb{E}\mathrm{X} \vert  \geq \varepsilon \mathbb{E}\mathrm{X} \rbrace \leq 2 \exp(- \varepsilon^2 \mathbb{E}\mathrm{X} /3).
\end{equation*}
\end{lemma}

\begin{alignat*}{4}
& \mathbb{P} \left\lbrace L^{3} \left( v_{n}, v'_{n} \right)  > t \right\rbrace
  \overset{}{\leq}  \mathbb{P} \left\lbrace d^{\frac{1}{p}} 2^{d+1} \left|  \frac{\sum\limits_{\substack{i}}^{}\mathbf{N}(x_{i}) }{
n} -
 a_{n}  \right|  \right.\\
 & \ \ \ \ \ \ \ \ \ \ \ \ \ \ \ \ \ \ \ \ \ \ \left. + d^{\frac{1}{p}}  2^{d+1}\left| 
a_{n}  - \frac{2  \sum\limits_{\substack{ i}}^{} \mathrm{L}_{i}}{n } \right| +a'_{n}> t \right\rbrace  \\
&   \overset{}{\leq}  \mathbb{P} \left\lbrace  \left|  \sum\limits_{\substack{i}}^{}\mathbf{N}(x_{i}) - 
n a_{n} \right|   > \dfrac{n t}{d^{\frac{1}{p}} 2^{d+2}}\right\rbrace \\
&  \ \ \ \ \ \ \ \ \ \ + \mathbb{P} \left\lbrace   
\left| 
na_{n}  - 2  \sum\limits_{\substack{ i}}^{} \mathrm{L}_{i} \right| > \dfrac{n t}{d^{\frac{1}{p}} 2^{d+2}}-\dfrac{na'_{n} }{d^{\frac{1}{p}} 2^{d+1}}\right\rbrace. 
\end{alignat*}

 Let 
\begin{equation*}
\mathrm{A}=  \mathbb{P} \left\lbrace  \left| \sum\limits_{\substack{i}}^{}\mathbf{N}(x_{i}) - 
n a_{n} \right|  > \dfrac{n t}{d^{\frac{1}{p}} 2^{d+2}}\right\rbrace,
\end{equation*}
and
\begin{equation*}
\mathrm{B}=  \mathbb{P} \Big\lbrace   
\Big| 
na_{n}  - 2  \sum\limits_{\substack{ i}}^{} \mathrm{L}_{i} \Big| > \dfrac{n t}{d^{\frac{1}{p}} 2^{d+2}}-\frac{na'_{n} }{d^{\frac{1}{p}} 2^{d+1}}\Big\rbrace.
\end{equation*}

We first upper bound the term $\mathrm{A}$ using Lemma \ref{Chebyschev-inequality} and \ref{variance-edges}.

\begin{align*}
\mathrm{A}&= \mathbb{P} \left\lbrace \left|  \sum\limits_{\substack{i}}^{}\mathbf{N}(x_{i}) - 
n a_{n} \right|  > \dfrac{n t}{d^{\frac{1}{p}} 2^{d+2}}\right\rbrace\\
& =  \mathbb{P} \left\lbrace  \left|  \xi_{n} - 
\mathbb{E}\xi_{n} \right|  > \dfrac{n t}{d^{\frac{1}{p}} 2^{d+3}}\right\rbrace \\
&\leq  \dfrac{d^{\frac{2}{p}} 2^{2d+6}\mathrm{Var}(\xi_{n})}{n^2t^2}\leq   \dfrac{d^{\frac{2}{p}} 2^{2d+6}\left[ \theta^{(d)}+2\theta^{(d)} a_{n} \right]}{n^2t^2}.
\end{align*}

Next, we upper bound the term $\mathrm{B}$. 

\begin{alignat*}{4}
\mathrm{B} & =   \mathbb{P} \left\lbrace   
\left|
na_{n}  - 2  \sum\limits_{\substack{ i}}^{}  \mathrm{L}_{i} \right| > \dfrac{n t}{d^{\frac{1}{p}} 2^{d+2}}-\dfrac{na'_{n} }{d^{\frac{1}{p}} 2^{d+1}}\right\rbrace \\
&  \leq  n  \mathbb{P} \left\lbrace   
a_{n}- 2   \mathrm{L}_{i} >\dfrac{t}{d^{\frac{1}{p}} 2^{d+2}}-\dfrac{a'_{n} }{d^{\frac{1}{p}} 2^{d+1}}\right\rbrace \\
&\ \ \ \ \ \ \ \ \ \ +  n  \mathbb{P} \left\lbrace   
 2   \mathrm{L}_{i}-a_{n} >\dfrac{t}{d^{\frac{1}{p}} 2^{d+2}}-\dfrac{a'_{n} }{d^{\frac{1}{p}} 2^{d+1}}\right\rbrace \\
 &  \leq  n  \mathbb{P} \left\lbrace   
\left| a_{n}-    \mathrm{L}_{i} \right| >\dfrac{t}{d^{\frac{1}{p}} 2^{d+3}}-\dfrac{a'_{n} }{d^{\frac{1}{p}} 2^{d+2}}+\dfrac{a_{n}}{2}\right\rbrace \\
&\ \ \ \ \ \ \ \ \ \ \ +  n  \mathbb{P} \left\lbrace   
    \mathrm{L}_{i} >\dfrac{t}{d^{\frac{1}{p}} 2^{d+3}}-\dfrac{a'_{n} }{d^{\frac{1}{p}} 2^{d+2}} +\dfrac{a_{n}}{2}\right\rbrace\\
   &   \overset{(a)}{\leq}   n  \mathbb{P} \left\lbrace   
\left| a_{n}-   \mathrm{L}_{i}\right| >\dfrac{t}{d^{\frac{1}{p}} 2^{d+3}}+\dfrac{a_{n}(2-c)}{4}\right\rbrace \\
&\ \ \ \ \ \ \ \ \ \ \ \ \ +  n  \mathbb{P} \left\lbrace   
  \mathrm{L}_{i} >\dfrac{t}{d^{\frac{1}{p}} 2^{d+3}}+\dfrac{a_{n}(2-c)}{4}\right\rbrace.
 \end{alignat*}
           
Step (a) follows by applying Lemma \ref{bound} and $c=\left( 1+\frac{1}{2a_{n}^{1/d}}\right)^d$. Then, 
    \begin{alignat*}{4}
    \mathrm{B}  &  \leq \\
&   n  \mathbb{P} \left\lbrace   
\left|
\mathbb{E}  \mathrm{L}_{i}-    \mathrm{L}_{i} \right|>\dfrac{t}{d^{\frac{1}{p}} 2^{d+3}}+\dfrac{a_{n}(2-c)}{4}- 2\theta^{(d)} n\mathrm{M}_{n}\right\rbrace \\
& \ \ \ \ \ \ \ \ \ +  n  \mathbb{P} \left\lbrace   
    \mathrm{L}_{i} >\dfrac{t}{d^{\frac{1}{p}} 2^{d+3}}+\dfrac{a_{n}(2-c)}{4}\right\rbrace \\
    &  \leq  n  \mathbb{P} \left\lbrace   
\left|
\mathbb{E}  \mathrm{L}_{i}-  \mathrm{L}_{i} \right|>a_{n} \varepsilon \right\rbrace \\
& \ \ \ \ \ \ \ \ \ \ \ +  n  \mathbb{P} \left\lbrace   
    \mathrm{L}_{i} >\dfrac{t}{d^{\frac{1}{p}} 2^{d+3}}+\dfrac{a_{n}(2-c)}{4}\right\rbrace,
\end{alignat*}
where 
\begin{equation*}
\varepsilon= \left( \frac{t}{d^{\frac{1}{p}} 2^{d+2}a_{n}}+\frac{(2-c)}{4}- \frac{2\mathrm{M}_{n}}{r_{n}} \right).
\end{equation*}

We continue by letting

\begin{equation*}
\mathrm{B}_{1}=  \mathbb{P} \left\lbrace   
\left|
\mathbb{E}\mathrm{L}_{i}-  \mathrm{L}_{i} \right|>a_{n} \varepsilon \right\rbrace.
\end{equation*}

\begin{equation*}
\mathrm{B}_{2}= \mathbb{P} \left\lbrace   
  \mathrm{L}_{i} >\dfrac{t}{d^{\frac{1}{p}}2^{d+3}}+\dfrac{a_{n}(2-c)}{4}\right\rbrace.
\end{equation*}

 For $n$ sufficiently large and consequently $a_{n}$ sufficiently large, we have  $1 \leq c<2$ and $0<\varepsilon \leq \frac{3}{2}.$ Therefore, by applying Lemma \ref{Chernof-binomial}, we upper bound $\mathrm{B}_{1}$ as
 
\begin{alignat*}{4}
& \mathbb{P} \left\lbrace   
\left|
\mathbb{E}\mathrm{L}_{i}-  \mathrm{L}_{i} \right|>a_{n} \varepsilon \right\rbrace \\
& \leq \mathbb{P} \left\lbrace   
\left|
\mathbb{E}\mathrm{L}_{i}-  \mathrm{L}_{i} \right|> (a_{n}-2n\theta^{(d)}\mathrm{M}_{n}) \varepsilon  \right\rbrace\\
&\overset{}{\leq}   2 \exp\Big( \dfrac{-\varepsilon^2}{3} \left(a_{n}-2n\theta^{(d)}\mathrm{M}_{n}\right)   \Big).
\end{alignat*}

The last term $\mathrm{B}_{1}$ is upper bounded  by using the Chernoff bound in Lemma \ref{Chernoff}. 

The probability generating function of the binomial random variable $\mathrm{L}_{i}$ is given by

\begin{equation*}
\left[a \theta^{(d)}(r_{n}-2\mathrm{M}_{n})+1- \theta^{(d)}(r_{n}-2\mathrm{M}_{n})\right]^n.
\end{equation*} 

Therefore, for $n$ sufficiently large, $1 \leq c<2$  and $a\geq 1$, we have

\begin{alignat*}{4}
 & \mathrm{B}_{2} \leq \dfrac{\left[\theta^{(d)}(r_{n}-2\mathrm{M}_{n})(a-1)+1\right]^n }{a^{\left(\dfrac{t}{d^{\frac{1}{p}}2^{d+3}}+\dfrac{a_{n}(2-c)}{4}\right)}}.
\end{alignat*}

Finally, taking the upper bounds of $\mathrm{A}$ and $\mathrm{B}$ obtained  from the upper bounds of $\mathrm{B}_{1}$ and $\mathrm{B}_{2}$ all together, Theorem \ref{theorem-connectivity} follows.
 \qed

\section{Proof of Lemma \ref{corollary1:density}}
\label{app:lemma3}
In this appendix, we provide the eigenvalues of the adjacency matrix of the \gls{DGG} using the $\ell_{\infty}$-metric.

When $d=1$, the adjacency matrix $\mathbf{A}(\mathcal{D}_{n})$ of a {\gls{DGG}} in $\mathbb{T}^1$ with $n$ nodes is a circulant matrix. A well known result  appearing in {\cite{gray2006toeplitz}}, states that the eigenvalues of a circulant matrix are given by the \gls{DFT} of the first row of the matrix. When $d>1$, the adjacency matrix of a {\gls{DGG}} is no longer circulant but it is block circulant with $\mathrm{N}^{d-1}\times\mathrm{N}^{d-1}$ circulant blocks, each of size $\mathrm{N} \times \mathrm{N}$. The author in {\cite{nyberg2014laplacian}}, pages 85-87, utilizes the result in {\cite{gray2006toeplitz}}, and shows that the eigenvalues of  the adjacency matrix in $\mathbb{T}^{d}$ are found by taking the $d$-dimensional \gls{DFT} of an $\mathrm{N}^ {d}$  tensor of rank $d$ obtained from the first block row of the matrix $\mathbf{A}(\mathcal{D}_{n})$
\begin{equation}
\label{2dDFT}
\lambda_{m_{1},...,m_{d}}= \sum\limits_{\substack{h_{1},...,h_{d}=0}}^{\mathrm{N}-1}  c_{h_{1},...,h_{d}}\exp\left(-\dfrac{2\pi i }{\mathrm{N}} {\bf m.h} \right), 
\end{equation}
where {\bf m} and {\bf h} are vectors of elements $m_{i}$ and $h_{i}$, respectively, with $m_{1},...,m_{d} \in \lbrace 0, 1,...,\mathrm{N}-1 \rbrace$ and $c_{h_{1},...,h_{d}}$ defined as \cite{nyberg2014laplacian}, 

\begin{equation}
\label{eqq1}
c_{h_{1},...,h_{d}}=\left\{
    \begin{array}{ll}
       0 , &  \mathrm{for} \ \  k_{n} <h_{1} ,..., h_{d} \leq \mathrm{N}-k_{n}-1 \\
       & \mathrm{or} \  h_{1} ,...h_{d}=0,  \\
       1, & \mathrm{otherwise}.
    \end{array}
\right. 
\end{equation}

The eigenvalues of the block circulant matrix $\mathbf{A}(\mathcal{X}_{n})$ follow the spectral decomposition \cite{nyberg2014laplacian}, page 86,
\begin{equation*}
\mathbf{A}=\mathbf{F}^{H}\mathbf{\Lambda} \mathbf{F},
\end{equation*}
where $\mathbf{\Lambda}$ is a diagonal matrix whose entries are the eigenvalues of $\mathbf{A}(\mathcal{X}_{n})$, and $\mathbf{F}$ is the $d$-dimensional \gls{DFT} matrix. It is well known that when $d=1$, the \gls{DFT} of an $n\times n$ matrix is the matrix of the same size with entries
\begin{equation*}
\mathbf{F}_{m,k}=\dfrac{1}{\sqrt{n}} \exp\left(-2\pi i m k/n\right) \ \ \ \mathrm{for } \ \ m,k=\lbrace0,1,...,n-1\rbrace.
\end{equation*}
When $d>1$, the block circulant matrix $\mathbf{A}$ is diagonalized by the $d$-dimensional \gls{DFT} matrix $\mathbf{F}= \mathbf{F}_{\mathrm{N}_{1}} \bigotimes... \bigotimes \mathbf{F}_{\mathrm{N}_{d}}$, i.e., tensor product,  where  $\mathbf{F}_{\mathrm{N}_{d}}$ is the $\mathrm{N}_{d}$-point \gls{DFT} matrix.

Using $(\ref{2dDFT})$ and $(\ref{eqq1})$, the eigenvalues of $\mathbf{A}(\mathcal{D}_{n})$ in $\mathbb{T}^d$ are given by
\begin{alignat*}{4}
\lambda_{m_{1},...,m_{d}}&=\left[\sum\limits_{\substack{h_{1},...,h_{d} =0}}^{\mathrm{N}-1} \exp\left(-\dfrac{2\pi i \bf{m} \bf{h}}{\mathrm{N}} \right) \right.\\
& \left. \ \ \ \ \  - \sum\limits_{\substack{h_{1},...,h_{d} =k_{n}+1}}^{\mathrm{N}-k_{n}-1} \exp\left(-\dfrac{2\pi i \bf{m} \bf{h}}{\mathrm{N}} \right) \right]-1\\
&= \sum\limits_{\substack{h_{1},...,h_{d}=k_{n}+1}}^{\mathrm{N}-k_{n}-1} \exp\left(-\dfrac{2\pi i \bf{m} \bf{h}}{\mathrm{N}} \right)- 1\\
&\overset{}{=} \prod_{s=1}^d    \dfrac{  \left( e^{\frac{-2im_{s} \pi}{\mathrm{N}}k_{n}}-e^{\frac{2im_{s} \pi}{\mathrm{N}}(1+k_{n})}\right)}{\left( -1+e^{\frac{2im_{s}\pi}{\mathrm{N}}}\right) }-1\\
&\overset{}{=}  \prod_{s=1}^d    \dfrac{  \left( e^{\frac{2im_{s} \pi}{\mathrm{N}}(1+k_{n})}-e^{\frac{-2im_{s} \pi}{\mathrm{N}}k_{n}}\right)}{\left( -1+e^{\frac{2im_{s}\pi}{\mathrm{N}}}\right) }- 1\\
&= \prod_{s=1}^d   \dfrac{\sin(\frac{m_{s} \pi}{\mathrm{N}}(2k_{n}+1))}{\sin(\frac{m_{s} \pi}{\mathrm{N}})}  - 1\\
&= \prod_{s=1}^d   \dfrac{\sin(\frac{m_{s} \pi}{\mathrm{N}}(a'_{n}+1)^{1/d})}{\sin(\frac{m_{s} \pi}{\mathrm{N}})}  - 1.
\end{alignat*}
  \qed

\end{document}